\documentclass[reqno]{amsart}
\usepackage[all]{xy}

\SelectTips{cm}{}

\usepackage{tikz}

\usetikzlibrary{shapes}

\usetikzlibrary{scopes}
\usetikzlibrary{calc}

\newcommand*{\usize}{1} %
\pgfmathsetmacro{\husize}{\usize/2} %

\newcommand{\mult}[2]
{
\coordinate (bot) at (#1,#2);
\coordinate (top) at ([shift={(0,\usize)}] bot);
\coordinate (mul) at ([shift={(\usize,\husize)}] bot);
\coordinate (end) at ([shift={(\usize,0)}] mul);

\draw (bot) .. controls +(\husize,0) .. (mul);
\draw (top) .. controls +(\husize,0) .. (mul);
\draw (mul) to (end);
}

\newcommand{\unit}[2]
{
\coordinate[fill=black,regular polygon, regular polygon sides=3, rotate=90, scale=0.5] (tr) at (#1,#2);
}

\newcommand{\point}[2]
{
\coordinate[fill=black, circle, scale=0.5] (circ) at (#1,#2);
}

\newcommand{\rec}[2]
{
\coordinate[fill=green, rectangle, scale=0.5] (rec) at (#1,#2);
}

\newcommand{\stepaside}[3]
{
\coordinate (beg) at (#1,#2);
\coordinate (end) at ([shift={(\usize*2,#3*\husize)}] beg);

\draw (beg) .. controls +(right:\usize) and +(left:\usize) .. (end);
}

\newcommand{\straight}[2]
{
\draw (#1,#2) -- ++(\usize*2,0);
}

\newcommand{\waste}[2]
{

\draw[green] (#1,#2) -- ++(\usize*2,0);

}

\def\dom{{\mathrm{dom}}}
\def\cod{{\mathrm{cod}}}

\def\id{{\mathrm{id}}}

\def\can{{\mathsf{can}}}

\def\tensor{\otimes}

\def\Lim#1#2{\{#1,#2\}}
\def\Colim#1#2{{#1}\star {#2}}
\def\colim{\mathop{\mathrm{colim}}}

\def\Lan#1#2{{\mathrm{Lan}}_{#1}#2}

\def\wt{\widetilde}
\def\wh{\widehat}

\def\Two{\mathbb{2}}
\def\One{\mathbb{1}}

\def\ol#1{\overline{#1}}

\def\less{\sqsubseteq}

\renewcommand{\phi}{\varphi}

\def\const{{\mathsf{const}}}

\def\class#1{\mathcal{#1}}
\def\Q{{\class{Q}}}
\def\P{{\class{P}}}

\def\DD{{\mathbb{D}}}

\def\kat#1{{\mathscr{#1}}}

\def\C{\kat{C}}
\def\E{\kat{E}}

\def\K{\kat{K}}

\def\D{\kat{D}}
\def\S{\kat{S}}
\def\T{\kat{T}}

\def\V{\kat{V}}

\def\op{{\mathit{op}}}

\def\elts#1{{\mathsf{elts}}(#1)}

\def\Set{{\mathsf{Set}}}
\def\Pos{{\mathsf{Pos}}}
\def\Pre{{\mathsf{Pre}}}
\def\Cat{{\mathsf{Cat}}}
\def\Cocts{{\mathsf{Cocts}}}
\def\Cont{{\mathsf{Cont}}}

\def\Alg{{\mathsf{Alg}}}

\renewcommand{\to}{\longrightarrow}

\usepackage[mathcal]{euscript}
\usepackage{amsbsy,mathrsfs,latexsym,amssymb,mathbbol}

\theoremstyle{plain}
\newtheorem{theorem}{Theorem}[section]
\newtheorem{proposition}[theorem]{Proposition}
\newtheorem{corollary}[theorem]{Corollary}
\newtheorem{lemma}[theorem]{Lemma}

\theoremstyle{definition}
\newtheorem{definition}[theorem]{Definition}
\newtheorem{example}[theorem]{Example}

\newtheorem{remark}[theorem]{Remark}

\numberwithin{equation}{section}

\usepackage[colorlinks,pagebackref,citecolor=blue]{hyperref}
\begin{document}
\title[An elementary characterisation of sifted weights]
      {An elementary characterisation of sifted weights}
\author{Mat\v{e}j Dost\'{a}l}
\address{Department of Mathematics, Faculty of Electrical Engineering, Czech Technical University
         in Prague, Czech Republic}
\email{dostamat@math.feld.cvut.cz}
\author{Ji\v{r}\'{\i} Velebil}
\address{Department of Mathematics, Faculty of Electrical Engineering, Czech Technical University
         in Prague, Czech Republic}
\email{velebil@math.feld.cvut.cz}
\thanks{Mat\v{e}j Dost\'{a}l acknowledges the support
        by the grant No. SGS14/186/OHK3/3T/13 
        of the Grant Agency of the Czech Technical University in Prague.
        Ji\v{r}\'{\i} Velebil acknowledges the support
        by the grant  No.~P202/11/1632
        of the Czech Science Foundation.} 
\date{9 May 2014}

\begin{abstract}
Sifted colimits (those that commute with finite products in sets) 
play a major r\^{o}le in categorical universal algebra. 
For example, varieties of (many-sorted) algebras are precisely the 
free cocompletions under sifted colimits 
of (many-sorted) Lawvere theories. 
Such a characterisation does not depend on the existence of finite
products in algebraic theories, but on the above fact that these 
products commute with sifted colimits and another condition:
finite products form a \emph{sound} class of limits. 

In this paper we study the notion of soundness for general classes
of weights in enriched category theory. We show that soundness 
of a given class of weights is equivalent to having a `nice'
characterisation of \emph{flat} weights for that class. 
As an application, we give an elementary characterisation
of sifted weights for the enrichment in categories and in
preorders. We also provide a number of examples of sifted 
weights using our elementary criterion. 
\end{abstract}

\maketitle

\section{Introduction}
\label{sec:intro}
The classical theory of locally presentable~\cite{gu} 
and accessible categories~\cite{lair}, \cite{makkai+pare} 
(see also the more recent~\cite{ar})
has been generalised to locally $\DD$-presentable and $\DD$-accessible
categories for a `good' class $\DD$ of small categories in~\cite{ablr}.
The classical theory hinges a lot upon the interplay of two classes
of categories: the class of $\lambda$-small categories for limits 
and the class of $\lambda$-filtered categories for colimits,
where $\lambda$ is a fixed regular cardinal.
The precise nature of the interplay is that
\begin{itemize}
\item[] 
$\lambda$-small limits commute with $\lambda$-filtered 
colimits in the category of all sets and mappings.
\end{itemize}
The idea of~\cite{ablr} was to develop a more general theory
of locally presentable and accessible categories
based on the fact that one has a fixed class $\DD$ of small
categories that replaces the class of $\lambda$-small 
categories. The corresponding class of colimits, 
called {\em $\DD$-filtered\/}, is then {\em defined\/}
by the requirement that
\begin{itemize}
\item[] 
$\DD$-limits commute with $\DD$-filtered colimits in the category
of sets and mappings.
\end{itemize}
It has been showed in~\cite{ablr} that a great deal of the classical
theory can be developed for the concept of $\DD$-filteredness, 
provided that the class $\DD$ satisfies
a side condition that is called {\em soundness\/}.

Roughly speaking, sound classes $\DD$ 
allow for an easier detection of $\DD$-filteredness: for a sound
class $\DD$ a category $\E$ is $\DD$-filtered if the process 
of taking colimits over $\E$ commutes with taking $\DD$-limits 
of {\em representable functors\/}.

For example, the class $\DD$ consisiting of finite discrete
categories is sound. The corresponding $\DD$-filtered
colimits turn out to be precisely the {\em sifted\/} colimits
of~\cite{lair:sifted}. Free cocompletions of small categories
under sifted colimits generalise the notion of a variety, as shown
in~\cite{adamek+rosicky:sifted}. In fact, the notion of a sifted
colimit turned out to be a cornerstone notion in the categorical
treatment of universal algebra, see~\cite{adamek+rosicky+vitale}.

One can pass from categories to categories enriched in a suitable
monoidal $\V$ and ask whether the above results can be reproduced.
Since (co)limits over a class of categories have to be replaced
by {\em weighted\/} (co)limits for $\V$-categories, one is 
naturally forced to define and study 
{\em soundness of classes of weights\/}. Definitions of soundness
of a class of weights have appeared in~\cite{lack+rosicky}
(Axiom~A) and in Section~4 of ~\cite{day+lack:small-functors}. 
The provided definitions of soundness, when specialised 
to categories enriched in sets, are, however, {\em weaker\/}
than the definition in~\cite{ablr}.

\subsection*{Results of the paper}
For categories enriched in a general $\V$, we give in Section~\ref{sec:soundness} a definition of soundness of a class $\Psi$ of weights 
that is equivalent to the notion of~\cite{ablr}
when $\V$ is the category of sets. We show in Proposition~\ref{prop:soundness-for-saturated} that all
notions of soundness that have appeared in the literature
coincide when the class $\Psi$
satisfies two side conditions that often arise in practice.
Namely, the class $\Psi$ has to be {\em locally small\/} 
in the sense of~\cite{kelly+schmitt} and 
{\em saturated\/} in the sense of~\cite{albert+kelly}
(there called a {\em closed\/} class).

Our definition allows us to give, for a sound class $\Psi$, 
a characterisation 
of {\em $\Psi$-flat weights\/} (those weights $\phi$
such that $\phi$-colimits commute with $\Psi$-limits
in $\V$) by means of a certain coend. We show that 
this characterisation,
when $\V$ is the category of sets and $\DD$ is sound, 
boils down to the characterisation
of $\DD$-filteredness of~\cite{ablr} 
in terms of cocones for certain diagrams.

In Section~\ref{sec:elementary} we turn the coend characterisation to a useful
criterion of $\Psi$-flat weights when the sound class $\Psi$
is the class of weights for finite products. Thus
$\Psi$-flat weights are precisely the (enriched) {\em sifted\/}
weights. Specialising to the enrichment in $\Cat$,  
we can therefore deal with siftedness for 2-categories. 
We apply the coend criterion to give elementary proofs
of siftedness of various weights used in 2-dimensional
universal algebra, see, e.g.~\cite{bourke:thesis}.

\subsection*{Acknowledgement}
We thank John Bourke for the discussions
concerning various aspects of siftedness. 

\section{Preliminaries on classes of weights}
\label{sec:preliminaries}

We introduce now the basic notation and results 
on weighted limits and colimits that we will need later.
The material here is standard; for more details 
we refer to the book~\cite{kelly:book} and the 
paper~\cite{kelly+schmitt}.

We fix a complete and cocomplete symmetric monoidal category
$\V=(\V_o,\tensor,I,[-,-])$. All categories, functors and
natural transformations are to be understood as $\V$-categories,
$\V$-functors and $\V$-natural transformations.

\subsection*{Limits and colimits}
A {\em weight\/} is a functor $\phi:\E^\op\to\V$, where $\E$
is small. Given a weight $\phi:\E^\op\to\V$ and a diagram
$D:\E\to\K$, a {\em colimit\/} of $D$ weighted by $\phi$
is an object $\Colim{\phi}{D}$ together with an isomorphism
$$
\K(\Colim{\phi}{D},X)
\cong
[\E^\op,\V](\phi,\K(D-,X))
$$ 
natural in $X$. Dually, a {\em limit\/} of a diagram
$D:\E^\op\to\K$ weighted by $\phi:\E^\op\to\V$ is an object
$\Lim{\phi}{D}$ together with an isomorphism 
$$
\K(X,\Lim{\phi}{D})
\cong
[\E^\op,\V](\phi,\K(X,D-))
$$
natural in $X$.

Given a weight $\phi:\E^\op\to\V$,
a category $\K$ is called {\em $\phi$-cocomplete\/}
if it has colimits $\Colim{\phi}{D}$ of all diagrams $D:\E\to\K$.
Analogously, we define {\em $\phi$-completeness\/}
of $\K$. 

Suppose $\K$ and $\kat{L}$ have $\phi$-colimits.
A functor $H:\K\to\kat{L}$ between $\phi$-cocomplete 
categories {\em preserves\/} 
$\phi$-colimits (or, is {\em $\phi$-cocontinuous\/})
if the canonical comparison 
$\Colim{\phi}{HD}\to H(\Colim{\phi}{D})$
is an isomorphism, for all $D:\E\to\K$. 

Even more generally, given a class $\Phi$
of weights, we say that $\K$ is {\em $\Phi$-cocomplete\/}
({\em $\Phi$-complete\/}, resp.) if it has $\phi$-colimits
($\phi$-limits, resp.) for all $\phi$ in $\Phi$. Analogously,
we define {\em $\Phi$-cocontinuous\/} 
({\em $\Phi$-continuous\/}, resp.) functors.

\subsection*{Free cocompletions and saturated classes of weights}
Given a class $\Phi$ of weights, let us write
$U_\Phi:\Phi\mbox{-}\Cocts\to\Cat$ for
the forgetful 2-functor from the 2-category  
$\Phi\mbox{-}\Cocts$ of $\Phi$-cocomplete categories, 
$\Phi$-continuous functors and natural transformations.
Then $U_\Phi$ has a left adjoint pseudofunctor, yielding
a {\em free $\Phi$-cocompletion\/} $\Phi(\K)$
of any category $\K$.

For a small category $\E$, the free $\Phi$-cocompletion 
$\Phi(\E)$ can be computed via a transfinite process, namely
the closure of $\E$ in $[\E^\op,\V]$ under 
$\Phi$-colimits.\footnote{In fact, the same transfinite
process can be applied to obtain $\Phi(\K)$ for 
{\em any\/} category $\K$. There is, however, a slight 
technicality concerning size when $\K$ is not small.
Since we will not need $\Phi(\K)$ for large $\K$, we
refer to~\cite{kelly:book} for more details.}
We will need the first step $\Phi_1(\E)$ of this process:
$\Phi_1(\E)$ is the full subcategory of $[\E^\op,\V]$
spanned by $\Phi$-colimits of representables.
Hence a weight $\alpha:\E^\op\to\V$ is in 
$\Phi_1(\E)$ iff it is of the form $\Lan{T^\op}{\phi}$ for
some $T:\D\to\E$ and some weight $\phi:\D^\op\to\V$ in $\Phi$. 

The class $\Phi$ is called {\em saturated\/} 
(the concept introduced in~\cite{albert+kelly},
there called {\em closed\/}) if, for any small category $\E$, the free cocompletion $\Phi(\E)$
consists precisely of all the weights $\phi: \E^\op \to \V$
that belong to $\Phi$. If we put $\Phi^*$ to be the largest class such that
the 2-categories $\Phi^*\mbox{-}\Cocts$ and $\Phi\mbox{-}\Cocts$
coincide, then $\Phi^*$ is saturated and it is the least
saturated class containing $\Phi$.

\subsection*{Commutation of limits and colimits, flatness}
Let $\Phi$ and $\Psi$ be classes of weights.
We say that {\em $\Phi$-colimits commute with
$\Psi$-limits in $\V$\/}, if for any $\phi:\E^\op\to\V$ 
in $\Phi$, the functor
$$
\Colim{\phi}{(-)}:[\E,\V]\to\V
$$
preserves $\Psi$-limits.
We denote by $\Psi^+$ the class of {\em $\Psi$-flat\/} 
weights, i.e., all weights $\phi$ such that
$\phi$-colimits commute with 
$\Psi$-limits in $\V$.

\begin{example}
\label{ex:Psi_D}
Suppose $\V$ is cartesian closed. By $\const_1:\D^\op\to\V$
we denote the weight that is constantly the terminal object $1$.
Such weights will be called {\em conical\/}.
Any class $\DD$ of small categories induces a class 
$$
\Psi_\DD
$$ 
of conical weights $\const_1:\D^\op\to\V$ with $\D^\op$
in $\DD$.
\begin{enumerate}
\item 
Suppose $\V=\Set$. Then to say that a small 
category $\E$ is {\em $\DD$-filtered\/} in the sense of~\cite{ablr}
is to say that the conical weight $\const_1:\E^\op\to\Set$ is 
$\Psi_\DD$-flat. Indeed: (co)limits of diagrams 
weighted by conical weights yield the usual notions
defined by (co)cones. 
\item
Suppose $\V$ is arbitrary (but still cartesian closed). The class
$\Psi_\DD$ for $\DD$ consisting of all finite discrete
categories will be denoted by
$\Pi$. The corresponding class of 
$\Pi$-flat weights is called the class of {\em sifted\/} weights.
We will say more on sifted weights in Section~\ref{sec:elementary}.
\end{enumerate} 
\end{example}

\section{Sound classes of weights}
\label{sec:soundness}

In this section we generalise the definition of a 
sound class $\DD$ of small (ordinary) categories to 
soundness of a class $\Psi$ of weights for a general $\V$.

Soundness of a class $\DD$ of small categories was
defined (in case $\V=\Set$) in~\cite{ablr} using 
connectedness of a certain category of cocones.
Since cocones are not available for general $\V$,
we use a different phrasing (already implicit in~\cite{ablr}). 
We prove in Proposition~\ref{prop:sound=ablr} 
below that our definition coincides with that of~\cite{ablr}
in case $\V=\Set$. 

For any $\Psi$-flat weight $\phi$, the functor 
$\Colim{\phi}{(-)}$ is obliged to preserve 
{\em all\/} $\Psi$-limits by the definition
of $\Psi$-flatness (see Section~\ref{sec:preliminaries}).
Soundness of $\Psi$ means that we can choose 
a {\em smaller\/}
class of $\Psi$-limits to detect $\Psi$-flatness.

\begin{definition}
\label{def:sound}
A class $\Psi$ of weights is called {\em sound\/}
if a weight $\phi:\E^\op\to\V$ is $\Psi$-flat whenever
the functor
$$
\Colim{\phi}{(-)}:[\E,\V]\to\V
$$ 
preserves $\Psi$-limits of representables.
\end{definition}

\begin{example}
\label{ex:sound-classes}
\mbox{}\hfill
\begin{enumerate}
\item 
In case $\V=\Set$, the list 
$$
\mbox{$\DD$=finite categories},
\quad
\mbox{$\DD$=finite discrete categories},
\quad
\mbox{$\DD$=empty class}
$$
yields a list of sound classes $\Psi_\DD$ of 
weights by Example~2.3 of~\cite{ablr}.

By the same example, the one-element class $\DD$ 
consisting of the scheme for pullbacks,
or the two-element class $\DD$ consisting of 
the scheme for pullbacks and terminal objects, yield
classes $\Psi_\DD$ that are not sound.
\item
It has been proved in~\cite{kelly+lack:strongly-finitary}
that the class $\Pi$ of weights for finite products
is sound, for every cartesian closed $\V$.
\item
The class $\P$ of all weights is sound for any $\V$. 
Indeed, the class of all $\P$-flat weights is precisely
the class $\Q$ of all small-projective weights by Proposition~6.20
of~\cite{kelly+schmitt}. By the same proposition,
small-projectivity of $\phi:\E^\op\to\V$ can be detected
by the fact that $\Colim{\phi}{(-)}:[\E,\V]\to\V$
preserves a particular limit of representables, namely 
the limit $\Lim{\phi}{Y}$, where $Y:\E^\op\to [\E,\V]$
is the Yoneda embedding. Thus, $\P$ is sound. 
\item
The class $\Q$ of all small-projective weights is sound
for any $\V$. By Remark~8.17 of~\cite{kelly+schmitt}, 
the class of $\Q$-flat weights
coincides with the class $\P$ of all weights. Hence
the condition on soundness is vacuous.
\end{enumerate}
\end{example}

The following easy result shows that the `testing weights'
for $\Psi$-flatness can be taken in a special form:

\begin{proposition}
\label{prop:characterisation-of-soundness}
For a class $\Psi$ the following are equivalent:
\begin{enumerate}
\item 
$\Psi$ is sound.
\item
The weight $\phi:\E^\op\to\V$ is $\Psi$-flat,
whenever $\Colim{\phi}{(-)}$ preserves $\Psi_1(\E)$-limits
of representables, i.e., whenever the canonical morphism
\begin{equation}
\label{eq:can}
\can:
\Colim{\phi}{\Lim{\psi}{Y-}}
\to
\Lim{\psi}{\phi}
\end{equation}
is an isomorphism, for every $\psi:\E^\op\to\V$ in $\Psi_1(\E)$.
\end{enumerate}
\end{proposition}
\begin{proof}
Let $Y:\E^\op\to [\E,\V]$ be the Yoneda embedding.
Definition~\ref{def:sound} requires the canonical
morphism 
$$
\Colim{\phi}{\Lim{\psi}{YT^\op-}}
\to
\Lim{\psi}{\phi\cdot T^\op}
$$
to be an isomorphism, for every $\psi:\D^\op\to\V$ in $\Psi$
and every $T:\D\to\E$. 

The weight $\Lan{T^\op}{\psi}:\E^\op\to\V$ is in $\Psi_1(\E)$
and every weight in $\Psi_1(\E)$ has this form, for some
$\psi:\D^\op\to\V$ in $\Psi$ and some $T:\D\to\E$.

Since there are isomorphisms
$$
\Lim{\psi}{YT^\op}
\cong
\Lim{\Lan{T^\op}{\psi}}{Y},
\quad
\Lim{\psi}{\phi\cdot T^\op}
\cong
\Lim{\Lan{T^\op}{\psi}}{\phi}
$$
the equivalence of (1) and (2) follows.
\end{proof}

The canonical morphism in~\eqref{eq:can} can be rewritten
using coends and Yoneda Lemma as the morphism
\begin{equation}
\label{eq:can-coend}
\can:
\int^e
[\E^\op,\V](Ye,\phi)\tensor [\E^\op,\V](\psi,Ye)
\to
[\E^\op,\V](\psi,\phi)
\end{equation}
that is given by composition in $[\E^\op,\V]$.
We illustrate now on two well-known classes that this 
coend description yields precisely
the `classical' description of flatness by means of 
the category of cocones.

\begin{example}[\bf Sifted weights and flat weights for $\V=\Set$]
\label{ex:sifted+filtered}
Suppose $\V=\Set$.
Recall that every $\phi:\E^\op\to\Set$ has a 
{\em category of elements\/} $\elts{\phi}$: the objects
are pairs $(x,e)$ with $x\in\phi e$ and a morphism
from $(x,e)$ to $(x',e')$ is a morphism $t:e\to e'$ in $\E$
such that $\phi t(e')=e$ holds.
\begin{enumerate}
\item
Let $\Pi$ be the sound class of weights for finite products. 
The category $\Pi_1(\E)$ is spanned by finite 
coproducts of representables in $[\E^\op,\Set]$. Hence
a general testing weight $\psi:\E^\op\to\Set$ for 
$\Pi$-flatness by Proposition~\ref{prop:characterisation-of-soundness} 
has the form $\coprod_{i\in I} Ye_i$ where $I$ is a finite set.

We show now that~\eqref{eq:can-coend}
yields the well-known characterisation of {\em sifted\/}
weights, see~\cite{lair:sifted}.

Indeed, given a general weight $\phi:\E^\op\to\Set$, the mapping
$\can$ has the form
$$
\can:
\int^e
\phi e \times \prod_{i\in I} \E(e_i,e)
\to
\prod_{i\in I} \phi e_i,
\quad
[(x,(t_i))]
\mapsto
(\phi t_i (x))
$$
Hence $\can$ is a bijection if and only if the following
two conditions hold:
\begin{enumerate}
\item 
The mapping $\can$ is surjective, i.e., for every 
element of $\prod_{i\in I} \phi e_i$, i.e., for every
$I$-tuple $(x_i)$ of elements of $\phi$ there is an $e$, an
element $x\in \phi e$ and an $I$-tuple $t_i:e_i\to e$
of morphisms in $\E$ such that $\phi t_i (x) = x_i$.

Briefly: on every $I$-tuple of objects of $\elts{\phi}$
there is a cocone.
\item
The mapping $\can$ is injective, i.e., for any pair
$(x,(t_i))$, $(x',(t'_i))$ such that $\phi t_i (x)=\phi t'_i (x')$
holds for all $i$, i.e., for any two cocones of the
same $I$-tuple of objects of $\elts{\phi}$, there
is a zig-zag in $\E$ that connects these cocones 
in $\elts{\phi}$.
\end{enumerate}
To summarise: a weight $\phi:\E^\op\to\Set$
is $\Pi$-flat iff its category of elements is sifted
(every finite family of elements has a cocone and every
two cocones for the same finite family are connected by
a zig-zag).
\item
Let $\Psi$ be the sound class of finite
(conical) limits, i.e., let $\Psi=\Psi_\DD$
for the class $\DD$ of finite categories. 

The category $\Psi_1(\E)$ is spanned by
finite colimits of representable functors 
in $[\E^\op,\Set]$.
Thus, a general testing weight $\psi:\E^\op\to\Set$
for $\Psi$-flatness has the form $\psi=\colim YC$ for a 
diagram $C:\C\to\E$ with $\C$ finite.

Given a general weight $\phi:\E^\op\to\Set$, the mapping
$\can$ is a bijection iff two conditions hold:
\begin{enumerate}
\item
The mapping $\can$ is surjective, i.e., every finite diagram
in $\elts{\phi}$ has a cocone.
\item
The mapping $\can$ is injective, i.e., any two cocones
for the same finite diagram in $\elts{\phi}$ are connected
by a zig-zag in $\elts{\phi}$.
\end{enumerate}
The above two conditions together state that the category of cocones
of finite diagrams in $\elts{\phi}$ is nonempty and connected.
This means that the category $\elts{\phi}$ is filtered. As expected,
$\Psi$-flat weights are precisely the flat ones.
\end{enumerate}
In both cases above, the classes of testing weights can be
simplified. For example, for siftedness, one can
choose only nullary coproduct of representables and binary
coproducts of representables as the testing weights. 
We will use this fact in Section~\ref{sec:elementary} below. 
\end{example}

We prove now that Definition~\ref{def:sound} coincides
with the definition of soundness from~\cite{ablr}
(this definition is condition~(2) of the proposition).

\begin{proposition}
\label{prop:sound=ablr}
Suppose $\V=\Set$. For a class $\DD$ of small categories,
the following conditions are equivalent:
\begin{enumerate}
\item 
The class $\Psi_\DD$ of conical weights $\const_1:\D^\op\to\Set$
with $\D^\op$ in $\DD$ is sound.
\item
A category $\E$ is $\DD$-filtered whenever the category 
of cocones for any functor $T:\D\to\E$ with $\D^\op$ in $\DD$
is nonempty and connected.
\end{enumerate}
\end{proposition}
\begin{proof}
We will use the canonical morphism~\eqref{eq:can-coend}.
Observe first that 
$[\E^\op,\Set](\psi,\const_1)$ is a one-element set
for any small category $\E$ and any $\psi:\E^\op\to\Set$, 
since $\const_1$ is a terminal object in $[\E^\op,\Set]$.

By Proposition~\ref{prop:characterisation-of-soundness} 
any testing weight $\psi:\E^\op\to\Set$ 
for $\Psi_\DD$-flatness of $\const_1:\E^\op\to\Set$
has the form $\Lan{T^\op}{\const_1}$ for some
$T:\D\to\E$, where $\const_1:\D^\op\to\Set$
is in $\Psi_\DD$. The left-hand side of~\eqref{eq:can-coend}
therefore has the form
$$
\int^e 
[\E^\op,\Set](\Lan{T^\op}{\const_1},Ye)
\cong
\int^e
[\D^\op,\Set](\const_1,Ye\cdot T^\op)
\cong
\int^e
[\D^\op,\Set](\const_1,\E(T-,e))
$$
Observe that the category of elements of 
$[\D^\op,\Set](\const_1,\E(T-,e))$ is precisely the category
of cocones for $T$ that have $e$ as a vertex.

Thus~\eqref{eq:can-coend} is a bijection iff 
$$
\int^e
[\D^\op,\Set](\const_1,\E(T-,e))
\cong
1
$$
holds. From this, the equivalence of~(1) and~(2) follows
immediately.
\end{proof}

There is another important issue related to sound classes
of weights. Adapting freely the terminology 
of~\cite{adamek+rosicky+vitale}, we may call a small
$\Psi$-complete category $\T$ a {\em $\Psi$-theory\/}.
The category $\Psi\mbox{-}\Alg(\T)$ of 
{\em $\Psi$-algebras for $\T$\/}
is the full subcategory of $[\T,\V]$ spanned by functors
that preserve $\Psi$-limits.

By definition, $\Psi$-flat colimits
commute with $\Psi$-limits. Hence
the category $\Psi\mbox{-}\Alg(\T)$ is closed in $[\T,\V]$ under
$\Psi$-flat colimits and it contains the representables, 
for any $\Psi$-theory $\T$.
Therefore, for any class $\Psi$ and any $\Psi$-theory $\T$, 
there is an inclusion 
$$
\Psi^+(\T^\op)\subseteq\Psi\mbox{-}\Alg(\T)
$$
since $\Psi^+(\T^\op)$ is the closure in $[\T,\V]$
of the representables under $\Psi$-flat colimits.
We discuss now the case when the above inclusion is an equality.

\begin{lemma}
\label{lem:sound=>nice-algebras}
Suppose the class $\Psi$ is sound. Then 
$\Psi^+(\T^\op)=\Psi\mbox{-}\Alg(\T)$ holds
for any $\Psi$-theory $\T$. 
\end{lemma}
\begin{proof}
Suppose $\phi:\T\to\V$ preserves $\Psi$-limits.
Then $\Colim{\phi}{(-)}:[\T^\op,\V]\to\V$ preserves
$\Psi$-limits of representables. By soundness of $\Psi$,
this means that $\phi$ is $\Psi$-flat. Hence
$\phi$ is in $\Psi^+(\T^\op)$. 
\end{proof}

\begin{remark}
The equality $\Psi^+(\T^\op)=\Psi\mbox{-}\Alg(\T)$
for any $\Psi$-theory $\T$ is an important fact
that allows for the development of an abstract theory
of `algebras' for $\Psi$. 

For example, when $\V=\Set$ and $\Psi=\Pi$, 
much of the theory of varieties of algebras hinges upon 
the above equality, see~\cite{adamek+rosicky+vitale}.
Namely, a $\Pi$-theory $\T$ is then precisely a (many-sorted) 
Lawvere theory, and a functor $\phi:\T\to\Set$ is an
algebra for $\T$ iff $\phi$ is a sifted weight. 

Another instance of the above is 
the equality, for any category $\T$ with finite limits,
of the free cocompletion ${\mathsf{Ind}}(\T^\op)$ 
of $\T^\op$ under filtered colimits and the category
${\mathsf{Lex}}(\T,\Set)$ of functors preserving finite
limits. This coincidence is vital in interpreting locally
finitely presentable categories as categories of algebras
for essentially algebraic theories, see~\cite{ar}.
\end{remark}

In Remark~2.6 of~\cite{ablr}, the authors present
the class $\Psi_\DD$ for $\DD$=(pullbacks+terminal object)
as the example of a class that is {\em not sound\/},
yet the equality $\Psi_\DD^+(\T^\op)=\Psi_\DD\mbox{-}\Alg(\T)$ 
holds. We show now that such a counterexample is essentially
due to the fact that $\Psi_\DD$ is not saturated.

In what follows we require a free 
$\Psi$-theory to exist on every small category.
More precisely, we require the class $\Psi$ to be 
{\em locally small\/} (see~\cite{kelly+schmitt}):
the category $\Psi(\D)$ is small for every small $\D$.

\begin{proposition}
\label{prop:soundness-for-saturated}
Suppose the class $\Psi$ is locally small and saturated.
Then the following are equivalent:
\begin{enumerate}
\item 
$\Psi$ is sound.
\item
$\Psi^+(\T^\op)=\Psi\mbox{-}\Alg(\T)$ holds
for any $\Psi$-theory $\T$.
\end{enumerate}
\end{proposition}
\begin{proof}
It suffices to prove that (2) implies (1).
Consider any weight $\phi:\E^\op\to\V$
such that the functor 
$$
F\equiv\Colim{\phi}{(-)}:[\E,\V]\to\V
$$
preserves $\Psi$-limits of representables.
We prove that $F$ preserves all $\Psi$-limits.

Denote by 
$$
\xymatrix{
\E^\op
\ar[0,1]_-{Z^\op}
&
\C^\op
\ar[0,1]_-{W}
&
[\E,\V]
\ar@{<-} `u[ll] `[ll]_-{Y_{\E^\op}} [ll]
}
$$  
the factorisation of the Yoneda embedding where
$Z^\op:\E^\op\to\C^\op$ is the closure of $\E^\op$
in $[\E,\V]$ under $\Psi$-limits. Notice that $\C^\op$
is small since $\Psi$ is locally small. 

By the construction, $\C^\op$ is a $\Psi$-theory and the functor
$W:\C^\op\to [\E,\V]$ preserves $\Psi$-limits. Furhtermore,
every object of $\C^\op$ is a $\Psi$-limit of representables,
since $\Psi$ is saturated. Hence the composite
$$
\xymatrix{
\ol{\phi}
\equiv
\C^\op
\ar[0,1]^-{W}
&
[\E,\V]
\ar[0,1]^-{F}
&
\V
}
$$
preserves $\Psi$-limits. By our assumption~(2),
the equality $\Psi^+(\C)=\Psi\mbox{-}\Alg(\C^\op)$
holds. Hence the functor $\Colim{\ol{\phi}}{(-)}:[\C,\V]\to\V$ 
preserves $\Psi$-limits.

Moreover,
$$
\Colim{\ol{\phi}}{(-)}\cong \Colim{(-)}{FW}
$$
holds, since $\Colim{\ol{\phi}}{(-)}\cong\Lan{Y_{\C^\op}}{\ol{\phi}}$,
and $\Lan{Y_{\C^\op}}{W}\cong\Colim{(-)}{W}$, and $F$ preserves
$\Lan{Y_{\C^\op}}{W}$ since $F=\Colim{\phi}{(-)}$ 
preserves colimits. See the diagram
$$
\xymatrix{
\C^\op
\ar[0,2]^-{Y_\C^\op}
\ar[1,2]_{W}
&
&
[\C,\V]
\ar[1,0]_{\Colim{(-)}{W}}
\ar `r[dd] `[dd]^-{\Colim{\ol{\phi}}{(-)}} [dd]
\\
&
&
[\E,\V]
\ar[1,0]_{F}
\\
&
&
\V
}
$$
In the adjunction $\Colim{(-)}{W}\dashv\wt{W}:[\E,\V]\to [\C,\V]$
the functor $\wt{W}$ is fully faithful, since $W$ is dense.
Thus, the composite
$$
\xymatrix{
[\E,\V]
\ar[0,1]^-{\wt{W}}
&
[\C,\V]
\ar[0,1]^-{\Colim{(-)}{W}}
&
[\E,\V]
\ar[0,1]^-{F}
&
\V
}
$$
is isomorphic to $F=\Colim{\phi}{(-)}$ and it preserves all 
$\Psi$-limits. We proved that $\phi$ is $\Psi$-flat;
the proof is finished.
\end{proof}

\begin{example}
Local smallness of a class $\Psi$ is a nontrivial property.
\begin{enumerate}
\item
Every small class is locally small. 
Example~15 of~\cite{velebil+adamek:conservative}
may be modified to show that, when $\V=\Set$, 
there exists a locally small class $\Psi$ 	
such that $\Psi$ is not contained in the saturation 
of any small class of weights.
\item
Every subclass of a locally small class is locally
small again. The saturation of a locally small class
is locally small again.
\item 
The class $\P$ of all weights is typically
not locally small. For example, 
$\P(\D)=[\D^\op,\Set]$ in case $\V=\Set$.
In fact, $\P$ is locally small iff $\V$
is small, i.e., iff $\V$ is a commutative quantale.
If $\V$ is a commutative quantale, then every class of weights
is locally small. 
\item
Even classes substantially smaller than $\P$ may not be
locally small. For example, consider the class $\Q$ of
small-projective weights that yields the Cauchy completion 
$\Q(\D)$ of the category $\D$, see~\cite{kelly+schmitt}.
\begin{enumerate}
\item
The class $\Q$ is locally small, whenever 
$\V$ is locally presentable as a monoidal category by 
Theorem~6 of~\cite{johnson:small-cauchy}.
\item
However, in case $\V$ is the monoidal closed category 
$\mathsf{Sup}$ of complete join-semilattices and 
join-preserving maps, the class $\Q$ 
is not locally small by Section~1 of~\cite{johnson:small-cauchy}. 
In this enrichment, the Cauchy completion ${\mathcal Q}(\D)$ 
of a small category $\D$ always contains 
all small coproducts of representables. 
\end{enumerate}
\end{enumerate}
\end{example}

\begin{remark}
We do not know whether local smallness of $\Psi$
can be omitted from the assumptions of 
Proposition~\ref{prop:soundness-for-saturated}.
\end{remark}

We can combine the above with Theorem~8.11 of~\cite{kelly+schmitt}
to obtain further characterisation of soundness
of locally small saturated classes.

\begin{corollary}
For a locally small saturated class $\Psi$, the following
conditions are equivalent:
\begin{enumerate}
\item 
$\Psi$ is sound.
\item
$\Psi^+(\T^\op)=\Psi\mbox{-}\Alg(\T)$ holds
for any $\Psi$-theory $\T$.
\item
For any small $\D$, every weight 
$\phi:\D^\op\to\V$ is a $\Psi$-flat colimit of
a diagram in $\Psi(\D)$.
\item
For any small $\D$, the closure of $\Psi(\D)$
in $[\D^\op,\V]$ under $\Psi$-flat colimits
is all of $[\D^\op,\V]$.
\end{enumerate}
\end{corollary}

\begin{remark}
The class $\Psi_\DD$ for $\DD$=(pullbacks+terminal object)
is locally small and not saturated. It is also not a 
sound class by Remark~2.6 of~\cite{ablr}. 
The saturation $\Psi_\DD^*$ is the class of finitely
presentable weights: this is a locally small, saturated 
and sound class.

In fact, one can always assume that $\Psi$ is 
a locally small saturated class, since it is easy 
to prove the following:
\begin{itemize}
\item[]
{\em If $\Psi$ is a locally small and sound class,
so is $\Psi^*$.\/} 
\end{itemize}
Indeed, if $\Psi$ is locally small, then so is $\Psi^*$.
Suppose $\Psi$ is sound. For proving soundness of $\Psi^*$, suppose $\T$ is a 
$\Psi^*$-theory.
By Lemma~\ref{lem:sound=>nice-algebras} we have 
the equality $\Psi^+(\T^\op)=\Psi\mbox{-}\Alg(\T)$.
Moreover, the equality
$\Psi\mbox{-}\Alg(\T)=\Psi^*\mbox{-}\Alg(\T)$ holds
by the definition of $\Psi^*$. Furthermore, the 
equality $\Psi^+=\Psi^{*+}$ holds by Proposition~5.4
of~\cite{kelly+schmitt}. Hence 
$\Psi^{*+}(\T^\op)=\Psi^*\mbox{-}\Alg(\T)$
holds for every $\Psi^*$-theory $\T$.
The class $\Psi^*$, being locally small and saturated, 
is sound by Proposition~\ref{prop:soundness-for-saturated}.
\end{remark}

\section{An elementary characterisation of sifted weights}
\label{sec:elementary}

In this section we analyse the isomorphism~\eqref{eq:can-coend}
in more detail for the enrichment in $\Cat$.
We then turn the analysis into a useful elementary criterion of 
{\em siftedness\/} of weights enriched in $\Cat$. Finally,
we comment on similarities and differences in using the 
criterion for siftedness in another `2-dimensional enrichment',
namely that in $\Pre$ (the category of preorders and monotone maps). 

\subsection*{An analysis of the coend in~\eqref{eq:can-coend}}
Suppose that $\V=\Cat$. Let $\psi,\phi:\E^\op\to\Cat$ 
be any weights. Then the coend
\begin{equation}
\label{eq:main-coend}
\int^e [\E^\op,\Cat](Ye,\phi)\times [\E^\op,\Cat](\psi,Ye)
\end{equation}
is a category
that can be computed as a coequaliser in $\Cat$ of the parallel
pair
\begin{equation}
\label{eq:L+R}
\vcenter{
\xymatrix{
\displaystyle\coprod_{e,e'} 
[\E^\op,\Cat](Ye',\phi)\times\E(e,e')\times [\E^\op,\Cat](\psi,Ye)
\ar@<.5ex> [0,1]^-{L}
\ar@<-.5ex> [0,1]_-{R}
&
\displaystyle\coprod_e 
[\E^\op,\Cat](Ye,\phi)\times [\E^\op,\Cat](\psi,Ye)
}
}
\end{equation}
of functors
$$
\begin{array}{l}
\\
L:
(\wh{x}:Ye'\to\phi,f:e\to e',\tau:\psi\to Ye)
\mapsto
(\wh{x}\cdot Yf:Ye\to\phi,\tau:\psi\to Ye)
\\
R:
(\wh{x}:Ye'\to\phi,f:e\to e',\tau:\psi\to Ye)
\mapsto
(\wh{x}:Ye'\to\phi,Yf\cdot \tau:\psi\to Ye')
\end{array}
$$
Thus the coend~\eqref{eq:main-coend} has the following desription 
(see, e.g., \cite{lawvere:thesis}):
\begin{enumerate}
\item 
The objects are equivalence classes 
$$
[(\wh{x},\tau)]_\sim
$$
where $\wh{x}:Ye\to\phi$ and $\tau:\psi\to Ye$
are natural transformations. The equivalence is generated
by 
$$
(\wh{x},Yf\cdot \tau)
\sim
(\wh{x}\cdot Yf,\tau)
$$
for all $\wh{x}:Ye\to\phi$, $\tau:\psi\to Ye'$
$f:e'\to e$ in $\D$.
\item
The morphisms are equivalence classes
$$
[((u_1,v_1),\dots,(u_n,v_n))]_\approx
$$
of finite sequences $((u_1,v_1),\dots,(u_n,v_n))$
such that every pair $(u_i,v_i)$ is a morphism in
the category 
$\coprod_e [\E^\op,\Cat](Ye,\phi)\times [\E^\op,\Cat](\psi,Ye)$
and 
$$
\cod (u_1,v_1) \sim \dom (u_2,v_2),
\quad
\cod (u_2,v_2) \sim \dom (u_3,v_3),
\quad
\dots,
\quad
\cod (u_{n-1},v_{n-1}) \sim \dom (u_n,v_n)
$$
The equivalence relation $\approx$ 
is generated from the following two conditions
$$
(u * Yw,v)\approx (u,Yw * v),
\quad
((u_1,v_1),(u_2,v_2))\approx (u_2\cdot u_1,v_2\cdot v_1)
$$
by reflexivity, symmetry, transitivity and composition
(concatenation).
Above, by $*$ we denote the horizontal composition
of natural transformations.
\end{enumerate}
It will be useful to work with the following 
graphical representation. The sequence
$((u_1,v_1),\dots,(u_n,v_n))$ as above is going to 
be depicted as
\begin{equation}
\label{eq:hammock}
\vcenter{
\xymatrix{
\phi
\ar@{=} [0,1]
&
\phi
\ar@{=} [0,1]
&
\phi
\ar@{}[0,1]|-{\cdots}
&
\phi
\ar@{=}[0,1]
&
\phi
\ar@{=}[0,1]
&
\phi
\\
e
\ar[-1,0]^{\wh{x}}
\ar@{~} [0,1]
&
e_1
\ar[-1,0]^{\wh{x}_1}
\ar@{=} [0,1]
\ar@{} [-1,1]|{\stackrel{u_1}{\Rightarrow}}
&
e_1
\ar[-1,0]_{\wh{x}'_1}
\ar@{}[0,1]|-{\cdots}
&
e_{n-1}
\ar[-1,0]^{\wh{x}_n}
\ar@{=} [0,1]
\ar@{} [-1,1]|{\stackrel{u_n}{\Rightarrow}}
&
e_{n-1}
\ar[-1,0]_{\wh{x}'_n}
\ar@{~} [0,1]
&
e'
\ar[-1,0]_{\wh{x}'}
\\
\psi
\ar[-1,0]^{\tau}
\ar@{=}[0,1]
&
\psi
\ar[-1,0]^{\tau_1}
\ar@{=} [0,1]
\ar@{} [-1,1]|{\stackrel{v_1}{\Rightarrow}}
&
\psi
\ar[-1,0]_{\tau'_1}
\ar@{}[0,1]|-{\cdots}
&
\psi
\ar[-1,0]^{\tau_n}
\ar@{=} [0,1]
\ar@{} [-1,1]|{\stackrel{v_n}{\Rightarrow}}
&
\psi
\ar@{=} [0,1]
\ar[-1,0]_{\tau'_n}
&
\psi
\ar[-1,0]_{\tau'}
}
}
\end{equation}
The above picture is called a {\em hammock\/} from
$(\wh{x},\tau)$ to $(\wh{x}',\tau')$. The wiggly
arrow in the above hammock, for example 
from $(\wh{x},\tau)$ to $(\wh{x}_1,\tau_1)$,
represents a zig-zag connecting $e$ and $e_1$ in $\E$
that witnesses the equivalence 
$(\wh{x},\tau)\sim (\wh{x}_1,\tau_1)$.

The whole hammock~\eqref{eq:hammock} gets evaluated to 
the composite modification
$$
(u_n*v_n)\cdot (u_{n-1}*v_{n-1})\cdot \dots \cdot (u_1*v_1):
\wh{x}\cdot\tau
\to
\wh{x}'\cdot\tau'
$$
in $[\E^\op,\Cat]$. Up to the equivalence $\approx$,
this is how the evaluation functor $\can$ works.

The functor $\can$ is an isomorphism of categories 
iff it is bijective on objects and fully faithful. 
Hence, the following two conditions have to hold:
\begin{enumerate}
\item 
{\em The 1-dimensional aspect.\/}
To give $\alpha:\psi\to\phi$ is to give
a unique $[(\wh{x},\tau)]_\sim$ such that
$\wh{x}\cdot\tau=\alpha$ holds.
\item
{\em The 2-dimensional aspect.\/}
To give a modification $\Xi:\alpha\to\alpha'$
is to give a unique equivalence class
$[((u_1,v_1),\dots,(u_n,v_u))]_\approx$
such that $\Xi$ is the
composite $(u_n * v_n)\cdot\dots\cdot (u_1* v_1)$. 
\end{enumerate}

\subsection*{Siftedness for enrichment in categories}
We are going to fix the class $\Pi$ of (conical)
weights for finite products, see Example~\ref{ex:Psi_D}.
Recall, by Example~\ref{ex:Psi_D} again, 
that the $\Pi$-flat weights are called sifted.

It is proved in~\cite{kelly+lack:strongly-finitary}
that the class $\Pi_1(\E)$ of testing weights 
for siftedness of a weight
$\phi:\E^\op\to\Cat$ can be reduced further
to the empty coproduct $\const_0:\E^\op\to\Cat$ 
of representables and to binary coproducts
$\E(-,e_1)+\E(-,e_2):\E^\op\to\Cat$. Hence, 
using~\eqref{eq:can-coend}, the following result holds:

\begin{lemma}
\label{lem:sifted=by-parts}
A weight $\phi:\E^\op\to\Cat$ is sifted iff
the following two conditions hold:
\begin{enumerate}
\item
The unique functor from $\int^e \phi e$ 
to the one-morphism category $\One$ is
an isomorphism.
\item
For any $e_1$, $e_2$ in $\E$, the canonical morphism 
$$
\can:\int^e \phi e\times\E(e_1,e)\times\E(e_2,e)
\to
\phi e_1
\times
\phi e_2
$$
is an isomorphism. 
\end{enumerate}
\end{lemma}

\begin{remark}
By analogy to the case $\V=\Set$, we may call the first
condition above {\em connectedness\/} of the weight 
$\phi:\E^\op\to\Cat$
and the second condition expresses that the diagonal
2-functor $\Delta:\E\to\E\times\E$ is {\em cofinal\/} in the
sense that the 2-cell
$$
\xymatrix{
\E^\op
\ar[0,2]^-{\Delta^\op}
\ar[1,1]_{\phi}
&
\ar@{}[1,0]|(.4){\stackrel{\delta}{\Rightarrow}}
&
\E^\op\times\E^\op
\ar[1,-1]^{\phantom{M}(e_1,e_2)\mapsto \phi e_1\times \phi e_2}
\\
&
\Cat
&
}
$$
where $\delta_e:\phi e\to \phi e\times \phi e$ is the diagonal
functor, is a left Kan extension. Indeed, it suffices to consider
the isomorphism
$$
\int^e \phi e\times \E(e_1,e)\times \E(e_2,e)
\cong
\int^e \phi e \times (\E^\op\times\E^\op)(\Delta^\op e,(e_1,e_2))
$$
\end{remark}

We apply the criteria of Lemma~\ref{lem:sifted=by-parts},
together with the analysis of~\eqref{eq:can-coend} using
hammocks, for giving elementary proofs of
siftedness of various weights. 

\begin{example}[\bf A weight that is not sifted]
\label{ex:sifted-but-not-cat-sifted}
We start with an example of a weight $\phi:\E^\op\to\Cat$
that is not sifted, although the `underlying' ordinary
functor 
$$
\xymatrix{
\E^\op_o
\ar[0,1]^-{\phi_o}
&
\Cat_o
\ar[0,1]^-{{\mathsf{ob}}}
&
\Set
}
$$
is sifted.

Consider the one-morphism category $\S$ with the only object 
$s$. Denote by $\E^\op$ the free completion of $\S$ under 
finite products. It follows immediately that the only 2-cells 
in $\E^\op$ are identities. 

Let $\chi:\E^\op\to\Cat$ be the product-preserving 
functor defined by $\chi(s) = \Two$, where $\Two$ is the
two-element chain, considered as a category. 
We define $\phi$ to be the following modification of $\chi$: 
where $\chi(s^n) = \Two^n$, we let $\phi(s^n) = 2^n$ 
for every $n > 1$. 
The structure on $2^n$ is that of an almost discrete 
preorder with the only 
nontrivial inequality being $(0,\dots,0) \leq (1,\dots,1)$. 
The action of $\phi$ on morphisms is defined as for $\chi$. 
Of course, $\phi$ does \emph{not} preserve products, but
the composite
$$
\xymatrix{
\E^\op_o
\ar[0,1]^-{\phi_o}
&
\Cat_o
\ar[0,1]^-{{\mathsf{ob}}}
&
\Set
}
$$
does; in fact, it is not hard to see that this ordinary 
functor constitutes an algebra for the ordinary algebraic theory 
$\E^\op_o$ and thus it is a sifted weight
by~\cite{adamek+rosicky+vitale}.

It is enough now to find pairs $(x_1,x_2)$ and 
$(y_1,y_2)$ from $\phi (s) \times \phi (s)$ such that 
$(x_1,x_2) \leq (y_1,y_2)$ holds but there is no hammock 
to witness this inequality. 
Consider $(x_1,x_2) = (0,1)$ and $(y_1,y_2) = (1,1)$. 
Firstly, we make use of the fact that there are no nontrivial
2-cells in $\E^\op$. This implies that the `lax' parts 
of the hammock consist only of inequalities between 
the elements of $\phi(s^n) = 2^n$ for some $s^n$. 
But these are precisely the diagonal inequalities 
$(0,\dots,0) \leq (1,\dots,1)$. Together with the fact 
that the only morphisms of the form $s^n \to s$ in $\E^\op$ 
are the product projections, it is easy to see that there 
is no way how any hammock could evaluate its right-hand side 
to $(1,1)$ and its left-hand side to $(0,1)$.
\end{example}

\begin{remark}
\label{rem:1-dim-vs-2-dim}
Siftedness of the composite 
${\mathsf{ob}}\cdot\phi_o:\E_o^\op\to\Set$
establishes precisely the 1-dimensional aspect
of siftedness: the functor $\can$ is bijective on
objects iff ${\mathsf{ob}}\cdot\phi_o$ is sifted.
From this it immediately follows that a weight
$\phi:\E^\op\to\Cat$ with $\E$ locally discrete
(i.e., with only the identity 2-cells) and such
that every $\phi e$ is a discrete category is sifted
iff the composite ${\mathsf{ob}}\cdot\phi_o:\E_o^\op\to\Set$
is sifted in the ordinary sense.

The 2-dimensional aspect of siftedness
of $\phi:\E^\op\to\Cat$ has to be verified in general.
Example~\ref{ex:sifted-but-not-cat-sifted}
exhibits such a situation when $\E$ is locally
discrete and Example~\ref{ex:conical-not-sifted}
shows a conical weight $\const_1:\E^\op\to\Cat$
that is not sifted although 
the underlying ordinary category $\E_o$ is sifted
in the ordinary sense.
\end{remark}

\begin{example}[\bf Siftedness for weights based on the simplicial
category]\label{ex:simpl-weight-sifted}
Recall from, e.g., \cite{maclane:cwm}, that the simplicial 
category $\Delta$ has finite ordinals as objects and monotone
maps as morphisms. It can be proved rather easily that the
morphisms of $\Delta$ can be obtained from $\id_1:1\to 1$,
$\eta:0\to 1$ and $\mu:2\to 1$ by ordinal sums subject to
monad axioms. Hence we will draw the morphisms of $\Delta$ 
as string diagrams that are generated from the following strings
$$
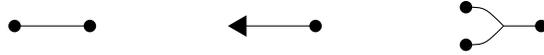
\begin{tikzpicture}[scale=.5]
\point{0}{0.5}
\point{2}{0.5}
\straight{0}{0.5}
\unit{6}{0.5}
\point{8}{0.5}
\straight{6}{0.5}
\mult{12}{0}
\point{12}{0}
\point{12}{1}
\point{14}{0.5}
\end{tikzpicture}
$$
that represent $\id_1:1\to 1$, $\eta:0\to 1$ and $\mu:2\to 1$,
respectively, by vertical concatenation that is 
subject to the unit axioms
$$
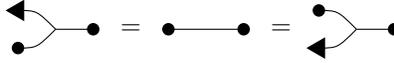
\begin{tikzpicture}[scale=.5]
\mult{0}{0}
\point{0}{0}
\unit{0}{1}
\point{2}{0.5}

\node at (3,0.5) {=};

{
	[shift={(4,0.5)}]
	
\straight{0}{0}
\point{0}{0}
\point{2}{0}
	
}
\node at (7,0.5) {=};

{
	[shift={(8,0)}]
	
\mult{0}{0}
\point{0}{1}
\unit{0}{0}
\point{2}{0.5}
}
\end{tikzpicture}
$$
and the associativity axiom
$$
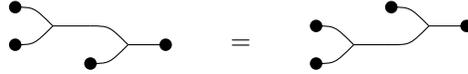
\begin{tikzpicture}[scale=.5]

\point{0}{0.5}
\point{0}{1.5}
\point{2}{0}
\point{4}{0.5}
\mult{0}{0.5}
\mult{2}{0}

\node at (6,0.5) {=};

{
	[shift={(8,0)}]
	
\point{0}{0}
\point{0}{1}
\point{2}{1.5}
\point{4}{1}
\mult{0}{0}
\mult{2}{0.5}

}

\end{tikzpicture}
$$

We show that both the conical weight on $\Delta$
and the weight given by inclusion of $\Delta$ into
$\Cat$ are sifted weights. In fact, from our reasoning 
it will be clear
that the same holds of almost any truncation $\Delta_n$.
The truncated category $\Delta_n$ 
is just the full subcategory of $\Delta$ spanned by 
finite ordinals up to $n$.
\begin{enumerate}
\item
It is known that $\const_1: \Delta \to \Set$ is an ordinary sifted weight, and therefore even the conical weight $\const_1: \Delta \to \Cat$ is sifted due to the fact that there are no non-trivial 2-cells in $\Delta$, see Remark~\ref{rem:1-dim-vs-2-dim}. 
Every truncation $\Delta_n$ (for $n \geq 1$) of the simplicial category $\Delta$ gives rise to a conical sifted weight as well.
\item
Suppose the weight $\phi: \Delta \to \Cat$ is given by inclusion.
Here 
$$
\xymatrix{
\Delta_o 
\ar[0,1]^-{\phi_o}
&
\Cat_o 
\ar[0,1]^-{\mathsf{ob}}
&
\Set
}
$$ 
is a representable weight $\Delta_o(1,-)$.
For each object $n$ of $\Delta$ the category $\phi(n)$ is the free linearly ordered category on an $n$-element chain.
We will show an elementary proof that $\phi$ is a sifted weight. First of all, let us check that the coend $\int^n \phi n$ is isomorphic to the one-morphism category $\One$. Of course, the category $\int^n \phi n$ has precisely one object: given any two objects $x \in \phi(n)$ and $y \in \phi(m)$, they are equivalent by $\sim$ if there exists a string diagram $\sigma: \phi(n) \to \phi(m)$ such that $x$ gets mapped to $y$ by $\sigma$. A diagram like this always exists; we illustrate this on an example situation with $n = 4$ and $m = 3$:
$$
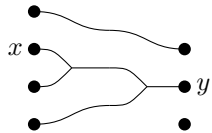
\begin{tikzpicture}[scale=0.5]

\point{0}{0}
\point{0}{1}
\point{0}{2}
\point{0}{3}
\node at (-0.5,2) {$x$};
\stepaside{0}{0}{1}
\mult{0}{1}
\mult{2}{0.5}
\stepaside{0}{3}{-1}
\stepaside{2}{2.5}{-1}

{
	[shift={(4,0)}]

\point{0}{0}
\point{0}{1}
\point{0}{2}
\node at (0.5,1) {$y$};
}
\end{tikzpicture}
$$
Now given any morphism $f: x \to x^\prime$ in $\phi(n)$, we show that $f \approx \id_*$, where $\id_*$ is the identity morphism on the only object $*$ of $\phi(1)$. This is again immediate when using the string diagrams: consider the only string diagram $!: \phi(n) \to \phi(1)$. It maps all morphisms in $\phi(n)$ to the identity morphism, see for example the diagram below.
$$
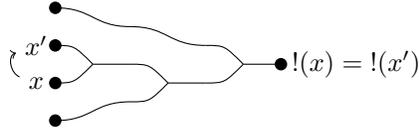
\begin{tikzpicture}[scale=0.5]

\point{0}{0}
\point{0}{1}
\point{0}{2}
\point{0}{3}
\node (x) at (-0.5,1) {$x$};
\node (xp) at (-0.5,2) {$x^\prime$};
\draw[out=160,in=200,->] (x) edge (xp);

\stepaside{0}{0}{1}
\mult{0}{1}
\mult{2}{0.5}
\stepaside{0}{3}{-1}
\stepaside{2}{2.5}{-1}
\mult{4}{1}

{
	[shift={(6,1.5)}]

\point{0}{0}

\node (y) at (2,0) {${!}(x) = {!}(x^\prime)$};

}
\end{tikzpicture}
$$
So the category $\int^n \phi n$ indeed has only one morphism. Now we show the isomorphism
$$
\int^n \phi n\times\Delta(n,n_1)\times\Delta(n,n_2)
\cong
\phi n_1
\times
\phi n_2
$$ 
by showing that the canonical morphism is bijective on objects and fully faithful. On objects, the canonical morphism takes an object $x \in \phi(n)$, two string diagrams $\sigma: \phi(n) \to \phi(n_1)$ and $\tau: \phi(n) \to \phi(n_2)$, and computes the pair $(\sigma(x),\tau(x))$. It is immediate that for any pair $(y,z)$ in $\phi n_1
\times
\phi n_2$ there exists a tuple $(x,\sigma,\tau)$ that is mapped to $(y,z)$. More is true: we can always choose $x = * \in \phi(1)$ and the string diagrams $\sigma,\tau$ are the obvious diagrams choosing $y$ and $z$, respectively.
$$
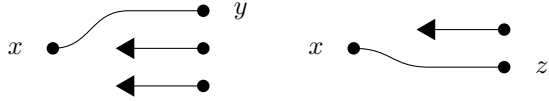
\begin{tikzpicture}[scale=0.5]
\node (x) at (-1,0) {$x$};
\point{0}{0}

\point{4}{-1}
\point{4}{0}
\point{4}{1}
\node (y) at (5,1) {$y$};
\stepaside{0}{0}{2}
\straight{2}{1}

\unit{2}{-1}
\unit{2}{0}
\straight{2}{-1}
\straight{2}{0}

{ [shift={(8,0)}]

\node (x) at (-1,0) {$x$};
\point{0}{0}

\point{4}{-0.5}
\point{4}{0.5}
\node (z) at (5,-0.5) {$z$};
\stepaside{0}{0}{-1}
\straight{2}{-0.5}

\unit{2}{0.5}
\straight{2}{0.5}

}
\end{tikzpicture}
$$
This proves that $\can$ is bijective on objects. In order to prove that $\can$ is full, we will show that given any pair of morphisms $g: y \to y^\prime$ and $h: z \to z^\prime$ in $\phi(m)$ and $\phi(p)$ respectively, there is a morphism $f: x \to x^\prime$ in $\phi(n)$ and two string diagrams sending the morphism $f$ to $g$ and $h$, respectively. But there is again a canonical such $f: x \to x^\prime$ in $\phi(2)$ with the obvious inclusions, as is shown in the example diagram below.
$$
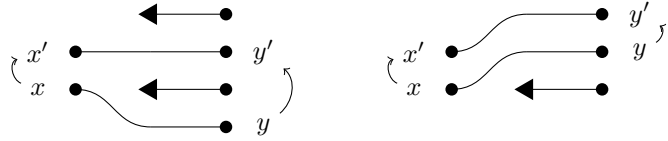
\begin{tikzpicture}[scale=0.5]

\point{0}{1}
\point{0}{2}

\node (x) at (-1,1) {$x$};
\node (xp) at (-1,2) {$x^\prime$};
\draw[out=160,in=200,->] (x) edge (xp);

\straight{0}{2}
\stepaside{0}{1}{-2}

\unit{2}{1}
\unit{2}{3}
\straight{2}{0}
\straight{2}{1}
\straight{2}{2}
\straight{2}{3}

{
	[shift={(4,0)}]

\point{0}{0}
\point{0}{1}
\point{0}{2}
\point{0}{3}
\node (y) at (1,0) {$y$};
\node (yp) at (1,2) {$y^\prime$};
\draw[out=40,in=-40,->] (y) edge (yp);
}

{
	[shift={(10,1)}]

\node (x) at (-1,0) {$x$};
\node (xp) at (-1,1) {$x^\prime$};
\draw[out=160,in=200,->] (x) edge (xp);

\point{0}{0}
\point{0}{1}

\point{4}{0}
\point{4}{1}
\point{4}{2}

\node (y) at (5,1) {$y$};
\node (yp) at (5,2) {$y^\prime$};
\draw[out=30,in=-30,->] (y) edge (yp);

\stepaside{0}{0}{2}
\straight{2}{1}

\unit{2}{0}
\straight{2}{0}

\stepaside{0}{1}{2}
\straight{2}{2}
}
\end{tikzpicture}
$$
Thus we have proved fullness \emph{and} faithfulness of the canonical functor $\can$. The weight $\phi$ is sifted.

We have actually proved that any truncation $\phi_n : \Delta_n \to \Cat$ of the inclusion weight is also sifted for $n \geq 2$.
\end{enumerate}
\end{example}

The 2-dimensional aspect of siftedness is crucial for $\Cat$-enriched weights even in the case of conical weights, as we show in the following easy example.

\begin{example}[\bf A conical weight that is not sifted]
\label{ex:conical-not-sifted}
Consider the diagram scheme for reflexive coequalisers
satisfying $\delta_0 \cdot \sigma = \delta_1 \cdot \sigma = \id_1$, 
and adjoin freely a 2-cell $\alpha$ to it:
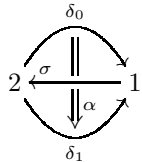
\[ 
\xy 
(-8,0)*+{2}="4"; 
(8,0)*+{1}="6"; 
{\ar_>>>{\sigma} "6";"4"}; 
{\ar@/^1.75pc/^{\delta_0} "4";"6"}; 
{\ar@/_1.75pc/_{\delta_1} "4";"6"}; 
{\ar@{=}^<<{} (0,6)*{};(0,1)*{}} ; 
{\ar@{=>}^<<<{\scriptstyle \alpha} (0,-1)*{};(0,-6)*{}} ; 
\endxy 
\]
The resulting 2-category $\E$, when considered as
a conical weight, is \emph{not} sifted, although
the underlying ordinary category $\E_o$ is sifted
in the ordinary sense (see, e.g., Chapter~3 of~\cite{adamek+rosicky+vitale}).
\end{example}

\begin{example}[\bf Siftedness for the weight for Kleisli objects]
\label{ex:kleisli}
The weight $\phi:\E^\op\to\Cat$ such that $\phi$-colimits yield Kleisli objects is described in~\cite{lawvere:ordinal-sums}. We will recall the definition of the weight $\phi$ and prove that it is sifted. That the weight $\phi$ is sifted is known from Proposition~8.43 in~\cite{bourke:thesis}: in this example we show an elementary proof of this fact.

The 2-category $\E$ is the suspension $\Sigma\Delta$ of the simplicial category $\Delta$. This means that $\E$ has a unique object, say $e_0$, and that the hom-category $\E(e_0,e_0)$ is the category $\Delta$. Morphisms in $\E$ are finite ordinals, and the 2-cells are `monad-like' string diagrams as described in Example~\ref{ex:simpl-weight-sifted}. 

The category $\phi(e_0)$ is defined as follows: the objects are finite \emph{non-zero} ordinals, that is, objects of the form $1 + n$ for some $n < \omega$. Every object $1 + n$ is understood as a $(n + 1)$-element chain with a distinguished bottom element. The morphisms in $\phi(e_0)$ are precisely the monotone maps that preserve the distinguished bottom element. This definition of $\phi(e_0)$ again allows a pictorial description in terms of string diagrams. The morphisms in $\phi(e_0)$ are string diagrams generated by the basic diagrams
$$
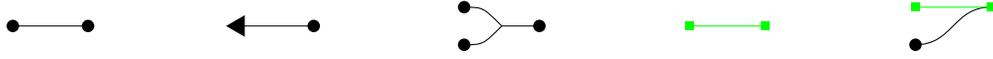
\begin{tikzpicture}[scale=.5]
\point{0}{0.5}
\point{2}{0.5}
\straight{0}{0.5}
\unit{6}{0.5}
\point{8}{0.5}
\straight{6}{0.5}
\mult{12}{0}
\point{12}{0}
\point{12}{1}
\point{14}{0.5}
\rec{18}{0.5}
\rec{20}{0.5}
\waste{18}{0.5}

\point{24}{0}
\stepaside{24}{0}{2} 
\rec{24}{1}
\rec{26}{1}
\waste{24}{1}
\end{tikzpicture}
$$
subject to monad axioms and the two axioms
$$
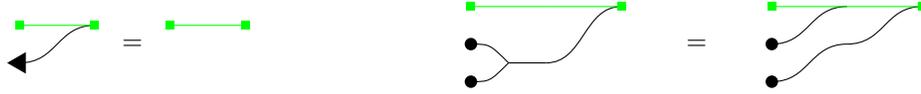
\begin{tikzpicture}[scale=0.5]

{
	[shift={(0,0.5)}]
	
\unit{0}{0}
\stepaside{0}{0}{2} 
\rec{0}{1}
\rec{2}{1}
\waste{0}{1}

\node at (3,0.5) {=};

\rec{4}{1}
\rec{6}{1}
\waste{4}{1}
}

{
	[shift={(12,0)}]

\mult{0}{0}
\point{0}{0}
\point{0}{1}
\stepaside{2}{0.5}{3}
\rec{0}{2}
\rec{4}{2}
\waste{0}{2} \waste{2}{2}

\node at (6,1) {=};

{
	[shift={(8,0)}]
	
\point{0}{0}
\point{0}{1}
\stepaside{0}{1}{2}
\stepaside{0}{0}{2}
\stepaside{2}{1}{2}

\rec{0}{2}
\rec{4}{2}
\waste{0}{2} \waste{2}{2}
}

}
\end{tikzpicture}
$$
that express the fact that the diagram 
\begin{tikzpicture}[scale=0.5,baseline=2]
\point{0}{0}
\stepaside{0}{0}{2} 
\rec{0}{1}
\rec{2}{1}
\waste{0}{1}
\end{tikzpicture}
is an algebra for the monad given by the unit 
\begin{tikzpicture}[scale=0.5,baseline=2]
\unit{0}{0.5}
\point{2}{0.5}
\straight{0}{0.5}

\node (tmp) at (-0.5,0.5) {};
\node (tmp2) at (2.5,0.5) {};
\end{tikzpicture}
 and multiplication
\begin{tikzpicture}[scale=0.5,baseline=2]
\mult{0}{0}
\point{0}{0}
\point{0}{1}
\point{2}{0.5}
\node (tmp) at (-0.5,0.5) {};
\node (tmp2) at (2.5,0.5) {};
\end{tikzpicture}.

The 2-functor $\phi: \E^\op \to \Cat$ is defined on the morphisms and 2-cells of $\E^\op$ by concatenation: for a given morphism $n: e_0 \to e_0$, the functor $\phi(n): \phi(e_0) \to \phi(e_0)$ maps an object $1 + m \in \phi(e_0)$ to the object $1 + m + n$. A string diagram $s$ in $\phi(e_0)$ is mapped to the diagram $\phi(n)(s)$, defined as the diagram $s$ concatenated $n$ identity strings. We show an example of this assignment for $n = 1$:
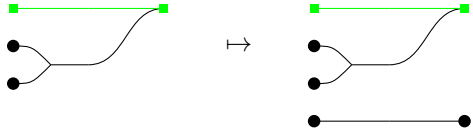
\begin{center}
\begin{tikzpicture}[scale=0.5]
\mult{0}{0}
\point{0}{0}
\point{0}{1}
\stepaside{2}{0.5}{3} 
\rec{0}{2}
\rec{4}{2}
\waste{0}{2} \waste{2}{2}

\node at (6,1) {$\mapsto$};

{
	[shift={(8,0)}]
	
\point{0}{-1}
\point{4}{-1}
\straight{0}{-1}
\straight{2}{-1}
\mult{0}{0}
\point{0}{0}
\point{0}{1}
\stepaside{2}{0.5}{3} 
\rec{0}{2}
\rec{4}{2}
\waste{0}{2} \waste{2}{2}
}
\end{tikzpicture}
\end{center}
Likewise, given a 2-cell $\theta: m \to n$ in $\E$, the natural transformation $\phi(\theta)$ is defined componentwise: for an object $1+ m$ in $\phi(e_0)$, the morphism $\phi(\theta)_{1+m}$ is the concatenation of the identity diagram on $1 + m$ with the diagram $\theta$. For example, given the diagram  \begin{tikzpicture}[scale=0.5,baseline=2]
\unit{0}{0.5}
\point{2}{0.5}
\straight{0}{0.5}

\node (tmp) at (-0.5,0.5) {};
\node (tmp2) at (2.5,0.5) {};
\end{tikzpicture} as $\theta$ and $m = 2$, the component $\phi(\theta)_3$ is the following string diagram in $\phi(e_0)$:
$$
\begin{tikzpicture}[scale=0.5]
\point{0}{1}
\point{0}{2}
\point{4}{2}
\point{4}{1}
\point{4}{0}

\unit{2}{0}

\rec{0}{3}
\rec{4}{3}
\waste{0}{3}
\straight{0}{2}
\straight{0}{1}
\waste{2}{3}
\straight{2}{2}
\straight{2}{1}
\straight{2}{0}
\end{tikzpicture}
$$

Now to prove that $\phi$ is a sifted weight, we need to verify that there are canonical isomorphisms
\begin{equation}\label{eq:canons-sifted}
\int^e \phi e \cong \One,
\qquad
\int^e \phi e\times\E(e_0,e)\times\E(e_0,e)
\cong
\phi e_0
\times
\phi e_0
\end{equation}
proving that $\Colim{\phi}{(-)}$ preserves nullary and binary products. We first analyse parts of a general hammock \eqref{eq:hammock} for the weight $\phi$ with the testing weight $\psi = \coprod_{i \in I} \E({-},e_i)$. The left-hand side rectangle on the diagram below
$$
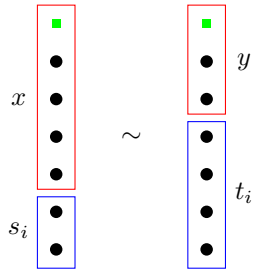
\begin{tikzpicture}
\xymatrix{
\phi
\ar@{=} [0,1]
&
\phi
\\
e_0
\ar[-1,0]^{x}
\ar@{->} [0,1]^{f}
&
e_0
\ar[-1,0]_{y}
\\
\psi
\ar[-1,0]^{s_i}
\ar@{=}[0,1]
&
\psi
\ar[-1,0]_{t_i}
}
\end{tikzpicture}
\qquad
\qquad
\qquad
\qquad
\begin{tikzpicture}[scale=0.5,baseline=(current bounding box.north)]

\node (x) at (-1,4) {$x$};
\node (si) at (-1,0.5) {$s_i$};

{ [red]
\draw (-0.5,1.6) rectangle (0.5,6.5);
}

{ [blue]
\draw (-0.5,-0.5) rectangle (0.5,1.4);
}

\point{0}{0}
\point{0}{1}
\point{0}{2}
\point{0}{3}
\point{0}{4}
\point{0}{5}
\rec{0}{6}

\node at (2,3) {$\sim$};

{
	[shift={(4,0)}]

\node (y) at (1,5) {$y$};
\node (ti) at (1,1.5) {$t_i$};

{ [red]
\draw (-0.5,3.6) rectangle (0.5,6.5);
}

{ [blue]
\draw (-0.5,-0.5) rectangle (0.5,3.4);
}

\point{0}{0}
\point{0}{1}
\point{0}{2}
\point{0}{3}
\point{0}{4}
\point{0}{5}
\rec{0}{6}	
}
\end{tikzpicture}
$$
represents the information that for each $i \in I$ and the morphisms given in the diagram we have that equalities $f + s_i = t_i$ and $x = y + f$ hold in natural numbers. This situation is depicted on the right-hand side of the above diagram. In general, the tuples $(x,s_i)$ and $(y,t_i)$ are related by the equivalence relation $\sim$ if and only if $x + s_i = y + t_i$ holds for all $i \in I$.

The rectangle of the form
$$
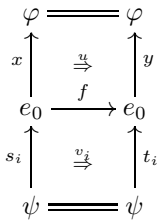
\xymatrix{
\phi
\ar@{=} [0,1]
&
\phi
\\
e_0
\ar[-1,0]^{x}
\ar@{->} [0,1]^{f}
\ar@{} [-1,1]|{\stackrel{u}{\Rightarrow}}
&
e_0
\ar[-1,0]_{y}
\\
\psi
\ar[-1,0]^{s_i}
\ar@{=}[0,1]
\ar@{} [-1,1]|{\stackrel{v_i}{\Rightarrow}}
&
\psi
\ar[-1,0]_{t_i}
}
$$
is represented by the concatenation of two string diagrams $u$ and $v_i$ for each $i \in I$.
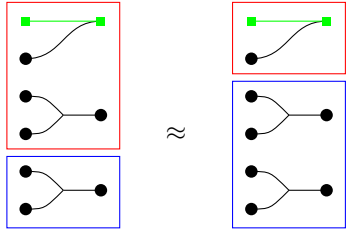
$$
\begin{tikzpicture}[scale=0.5]

{ [red]
\draw (-0.5,3.6) rectangle (2.5,7.5);
}

{ [blue]
\draw (-0.5,1.5) rectangle (2.5,3.4);
}

\point{0}{2}
\point{0}{3}
\point{0}{4}
\point{0}{5}
\point{0}{6}
\mult{0}{2}
\mult{0}{4}
\stepaside{0}{6}{2}
\waste{0}{7}
\rec{0}{7}
\rec{2}{7}

{ [shift={(2,0.5)}]

\point{0}{2}
\point{0}{4}
}

\node at (4,4) {$\approx$};

{
	[shift={(6,0)}]
{ [red]
\draw (-0.5,5.6) rectangle (2.5,7.5);
}

{ [blue]
\draw (-0.5,1.5) rectangle (2.5,5.4);
}

\point{0}{2}
\point{0}{3}
\point{0}{4}
\point{0}{5}
\point{0}{6}
\mult{0}{2}
\mult{0}{4}
\stepaside{0}{6}{2}
\waste{0}{7}
\rec{0}{7}
\rec{2}{7}

{ [shift={(2,0.5)}]

\point{0}{2}
\point{0}{4}
}
}
\end{tikzpicture}
$$
The above diagram is an example of string diagrams that are equivalent: the `sliding' of the division between the string diagrams generates the equivalence relation $\approx$. Observe moreover that morphisms in the coend are $n$-tuples of composable string diagrams. Any such $n$-tuple is equivalent to a $1$-tuple, but the fact that we are allowed to vertically `decompose' any string diagram to $n$ parts is important in the proof of siftedness for $\phi$. In the following diagram
$$
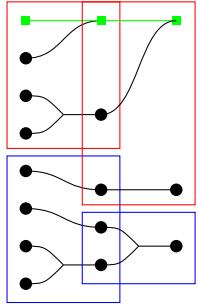
\begin{tikzpicture}[scale=0.5]

{ [red]
\draw (-0.5,3.6) rectangle (2.5,7.5);
}

{ [blue]
\draw (-0.5,-0.5) rectangle (2.5,3.4);
}

\point{0}{0}
\point{0}{1}
\point{0}{2}
\point{0}{3}
\point{0}{4}
\point{0}{5}
\point{0}{6}
\mult{0}{0}
\stepaside{0}{2}{-1}
\stepaside{0}{3}{-1}
\mult{0}{4}
\stepaside{0}{6}{2}
\waste{0}{7}\waste{2}{7}
\rec{0}{7}
\rec{2}{7}
\rec{4}{7}

{ [shift={(2,0.5)}]

{ [red]
\draw (-0.5,1.6) rectangle (2.5,7);
}

{ [blue]
\draw (-0.5,-0.5) rectangle (2.5,1.4);
}

\point{0}{0}
\point{0}{1}
\point{0}{2}
\point{0}{4}

\mult{0}{0}
\straight{0}{2}
\stepaside{0}{4}{5}

\point{2}{0.5}
\point{2}{2}
}
\end{tikzpicture}
$$
we can see such a decomposition of a string diagram into a $2$-tuple of shorter string diagrams.

With the complete description of the weight $\phi$ and of the hammocks, we can conclude that we have the canonical isomorphisms in~\eqref{eq:canons-sifted}:
\begin{enumerate}
\item The weight $\phi$ satisfies the isomorphism
$$\int^e \phi(e) \cong \One.$$
Indeed, the coend $\int^e \phi(e)$ has precisely one object: any pair $1 + n$ and $1 + m$ of objects in $\phi(e_0)$ is related by a hammock of length 2:
$$
\xymatrix{
\phi
\ar@{=} [0,1]
&
\phi
\ar@{=} [0,1]
&
\phi
\\
e_0
\ar[-1,0]^{1+n}
\ar[0,1]_{n}
&
e_0
\ar[-1,0]^{1}
&
e_0
\ar[-1,0]_{1+m}
\ar[0,-1]^{m}
}
$$
To show that $\int^e \phi(e)$ has a unique morphism, we will prove that any string diagram $\sigma: 1 + m \to 1 + n$ is congruent by the equivalence relation $\approx$ to an identity string diagram $\id_k: k \to k$ for some natural number $k$. We have to distinguish two cases. If the diagram $\sigma$ does not contain \begin{tikzpicture}[scale=0.5,baseline=2]
\point{0}{0}
\stepaside{0}{0}{2} 
\rec{0}{1}
\rec{2}{1}
\waste{0}{1}
\end{tikzpicture}
as a subdiagram, then it is trivially a concatenation of two string diagrams $\sigma_0 = \id_1: 1 \to 1$ and $\sigma_1: m \to n$, and therefore $\sigma \approx \id_1$ holds. If $\sigma$ contains \begin{tikzpicture}[scale=0.5,baseline=2]
\point{0}{0}
\stepaside{0}{0}{2} 
\rec{0}{1}
\rec{2}{1}
\waste{0}{1}
\end{tikzpicture}, then it is necessary to factor it into a composition of two diagrams (and denote the red part of the diagram by $\omega$):
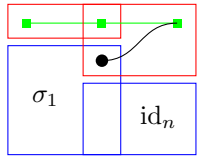
\begin{center}
\begin{tikzpicture}[scale=0.5]

{ [red]
\draw (-0.5,6.6) rectangle (2.5,7.5);
}

{ [blue]
\draw (-0.5,3.5) rectangle (2.5,6.4);
}

\waste{0}{7}\waste{2}{7}
\rec{0}{7}\rec{2}{7}\rec{4}{7}

\node at (0.5,5) {$\sigma_1$};

{ [shift={(2,0)}]

{ [red]
\draw (-0.5,5.6) rectangle (2.5,7.5);
}

{ [blue]
\draw (-0.5,3.5) rectangle (2.5,5.4);
}

\point{0}{6}
\stepaside{0}{6}{2}
\node at (1.5,4.5) {$\id_n$};

}
\end{tikzpicture}
\end{center}
This decomposition is unique. Take the identity morphism $\id_n$ and decompose it in the same way into a concatenation of $\omega$ with $(\tau_1,\id_n)$. By the first case we have that $\tau_1 \approx \sigma_1$, and the equivalence $\id_n \approx \id_n$ is trivial. This decomposition thus witnesses the equivalence $\sigma \approx \id_n$.
\item The second isomorphism
$$
\int^e \phi e\times\E(e_0,e)\times\E(e_0,e)
\cong
\phi e_0
\times
\phi e_0
$$
is proved similarly. Given two objects $1 + m$ and $1 + n$ from $\phi(e_0)$, there is a triple $(1,m,n)$ that gets mapped exactly to $(1 + m, 1 + n)$ by the canonical functor. For any other triple $(k,m^\prime,n^\prime)$ that is mapped to $(1 + m, 1 + n)$ we have the equalities $k + m^\prime = 1 + m$ and $k + n^\prime = 1 + n$. Therefore $(1,m,n) \sim (k,m^\prime,n^\prime)$ holds and the canonical functor is bijective on objects.

To prove that the canonical functor is full, we show that for any two string diagrams $\sigma: 1 + m \to 1 + n$ and $\tau: 1 + p \to 1 + q$ there is a triple $(\omega,\alpha,\beta)$ getting mapped to $(\sigma,\tau)$. But again, as in the case of the first isomorphism, take $\omega$ to be the diagram
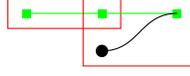
\begin{center}
\begin{tikzpicture}[scale=0.5]

{ [red]
\draw (-0.5,6.6) rectangle (2.5,7.5);
}

\waste{0}{7}\waste{2}{7}
\rec{0}{7}\rec{2}{7}\rec{4}{7}

{ [shift={(2,0)}]

{ [red]
\draw (-0.5,5.6) rectangle (2.5,7.5);
}

\point{0}{6}
\stepaside{0}{6}{2}

}
\end{tikzpicture}
\end{center}
and factor the diagrams $\sigma$ and $\tau$ into pairs $\alpha = (\sigma_1,\id_n)$ and $\beta = (\tau_1,\id_q)$ in a way that $\omega * \alpha = \sigma$ and $\omega * \beta = \tau$, where $*$ denotes the horizontal composition. Faithfulness of the canonical functor then comes easily from the fact that the morphisms in the coend have the above mentioned `normal form'.
\end{enumerate}
\end{example}

\subsection*{Siftedness for enrichment in preorders}
The enrichment in the category $\Pre$ of preorders
and monotone maps is in many aspects similar to the enrichment
in $\Cat$, but the computations are much simpler. In fact, we
will be able to give a full characterisation of sifted
{\em conical\/} weights $\const_1:\E^\op\to\Pre$, see
Example~\ref{ex:sifted-conical-Pre}.

The crucial coend
$$
\int^e \phi e\times [\E^\op,\Pre](\psi,Ye)
$$
is computed as a coequaliser in $\Pre$ of two monotone
maps $L$ and $R$ that are defined in the same way as for
$\V=\Cat$, see~\eqref{eq:L+R}. 
Moreover, the coequaliser of $L$ and $R$ 
can be computed in two steps.
First we compute the coequaliser on the level of 
underlying sets. This yields a set of equivalence
classes of the form $[(\wh{x},\tau)]_\sim$ w.r.t.
the equivalence $\sim$ generated by $L$ and $R$.
The set of equivalence classes is then equipped
with a least preorder $\less$ satisfying the 
following condition:
\begin{enumerate}
\item[] 
If $(\wh{x},\tau)\leq (\wh{y},\sigma)$,
then $[(\wh{x},\tau)]_\sim\less [(\wh{y},\sigma)]_\sim$.
\end{enumerate}
where $\leq$ denotes the preorder of the coproduct
$\coprod_e \phi e\times [\E^\op,\Pre](\psi,Ye)$.

Below, we will also use hammocks for the enrichment in 
$\Pre$. These are pictures like~\eqref{eq:hammock} but the 2-cells
$u_i$, $v_i$ are replaced by mere inequality signs.

We show now that for {\em conical\/} weights 
$\phi:\E^\op\to\Pre$ 
the 2-dimensional aspect of siftedness is vacuous.
That this is not true for {\em general\/} weights
$\phi:\E^\op\to\Pre$ is demonstrated by the weight of
Example~\ref{ex:sifted-but-not-cat-sifted}: all
categories there are in fact enriched in $\Pre$.

\begin{example}[\bf Sifted conical weights]
\label{ex:sifted-conical-Pre}
The reasoning is similar to Example~\ref{ex:sifted+filtered}
above. Elements of $\Psi_1(\E)$ are finite coproducts 
$\coprod_{i\in I}Ye_i$ 
of representables in $[\E^\op,\Pre]$. By Yoneda Lemma, every
$\tau:\psi\to Ye$ can be identified with a cocone
$t_i:e_i\to e$. Then the requirement that for any two natural
transformations $\tau:\psi\to Ye$ and $\sigma:\psi\to Ye$ 
the equivalence $\tau \sim \sigma$ has to hold, corresponds 
to the fact that the cocones  
$t_i:e_i\to e$ and $s_i:e_i\to e$ 
(corresponding to $\tau$ and $\sigma$ respectively) 
have to be connected by a zig-zag. The 2-dimensional
aspect of siftedness is vacuous in this case.

Thus a weight $\const_\One:\E^\op\to\Pre$ is sifted 
if and only if the \emph{ordinary} functor
$$
\xymatrixcolsep{2.5pc}
\xymatrix{
\E^\op_o
\ar[0,1]^-{(\const_\One)_o}
&
\Pre_o
\ar[0,1]^-{{\mathsf{ob}}}
&
\Set
}
$$
is sifted in the \emph{ordinary} sense.
\end{example}

\begin{example}[\bf Sifted weights in general]
Consider a general weight $\phi:\E^\op\to\Pre$. 
To establish the isomorphism 
$$
\can:
\int^e
\phi e \times \prod_{i\in I} \E(e_i,e)
\to
\prod_{i\in I} \phi e_i,
$$
of preorders we need the monotone map $\can$ 
to be bijective and order-reflecting. 
As we noticed earlier, the coend is computed as a coequaliser 
in $\Set$ equipped with a freely generated preorder. 
More precisely, there are two conditions for a weight to be sifted:
\begin{enumerate}
\item 
To obtain bijectivity of the $\can$ mapping we demand that
$$
\xymatrix{
\E^\op_o
\ar[0,1]^-{\phi_o}
&
\Pre_o
\ar[0,1]^-{{\mathsf{ob}}}
&
\Set
}
$$
be an ordinary sifted weight.
\item
Order-reflectivity of $\can$ means that given any two tuples 
$(x_i) \leq (x'_i)$ from $\prod_{i\in I} \phi e_i$ 
we can form a hammock
$$
\xymatrix{
\phi
\ar@{=} [0,1]
&
\phi
\ar@{=} [0,1]
&
\phi
\ar@{}[0,1]|-{\cdots}
&
\phi
\ar@{=}[0,1]
&
\phi
\ar@{=}[0,1]
&
\phi
\\
e
\ar[-1,0]^{\wh{x}}
\ar@{~} [0,1]
&
e_1
\ar[-1,0]^{\wh{x}_1}
\ar@{=} [0,1]
\ar@{} [-1,1]|{\leq}
&
e_1
\ar[-1,0]_{\wh{x}'_1}
\ar@{}[0,1]|-{\cdots}
&
e_{n-1}
\ar[-1,0]^{\wh{x}_n}
\ar@{=} [0,1]
\ar@{} [-1,1]|{\leq}
&
e_{n-1}
\ar[-1,0]_{\wh{x}'_n}
\ar@{~} [0,1]
&
e'
\ar[-1,0]_{\wh{x}'}
\\
(e_i)
\ar[-1,0]^{\tau}
\ar@{=}[0,1]
&
(e_i)
\ar[-1,0]^{\tau_1}
\ar@{=} [0,1]
\ar@{} [-1,1]|{\leq}
&
(e_i)
\ar[-1,0]_{\tau'_1}
\ar@{}[0,1]|-{\cdots}
&
(e_i)
\ar[-1,0]^{\tau_n}
\ar@{=} [0,1]
\ar@{} [-1,1]|{\leq}
&
(e_i)
\ar@{=} [0,1]
\ar[-1,0]_{\tau'_n}
&
(e_i)
\ar[-1,0]_{\tau'}
}
$$
such that its left-hand vertical side 
evaluates to $(x_i)$, and its right-hand vertical side 
evaluates to $(x'_i)$.
\end{enumerate}
\end{example}

\begin{remark}
\label{rem:pre-cat-sifted}
Observe that the characterisations of sifted weights for
enrichments in $\Cat$ and $\Pre$ are strongly related. 
This is because the computations of coequalisers are
essentially the same.

In fact, the requirements for a weight 
$\phi: \E^\op \to \Pre$ to be sifted (as enriched in $\Pre$) 
are exactly the requirements of siftedness for the weight 
$\phi^\prime: \E^\op \to \Cat$, with $\phi^\prime$ 
being the weight $\phi$ considered as enriched 
in $\Cat$.

The situation is rather different when considering
sifted weights for the enrichment in the category $\Pos$
of all posets and monotone maps. The computation of
a coequaliser in $\Pos$ runs in two steps: one computes
the coequaliser in preorders {\em and then\/} performs
the poset-reflection. It is the second step that brings
in additional identifications and makes the characterisation
of siftedness quite complex. 
\end{remark}

\section{Conclusions and future work}
\label{sec:conclusions}

We gave a generalisation of the concept of a sound class of
small categories~\cite{ablr} to a sound class $\Psi$ of weights in the
context of enriched category theory. When passing to classes of weights, we showed that an easy analysis of flatness w.r.t.\ a sound class of weights is possible, providing us, for example, with known characterisations of siftedness and filteredness of ordinary categories. The same result yields elementary proofs of siftedness of various weights for the enrichment in categories or preorders.

A further characterisation of soundness is essentially
contained in~\cite{day+lack:small-functors}. Namely,
a class $\Psi$ of weights is sound iff the KZ-monad 
$\K\mapsto\P(\K)$ of free cocompletions under all colimits
lifts to the 2-category $\Psi\mbox{-}\Cont$
of all $\Psi$-complete categories, all $\Psi$-continuous
functors and all natural transformations.
We thank John Bourke for pointing this out to us.

The theory of lex colimits 
of~\cite{garner+lack:lex-colimits} works
precisely due to the lifting of $\K\mapsto\P(\K)$ to
$\Psi\mbox{-}\Cont$ for $\Psi$ being the class of weights
for finite limits. Hence a theory of `colimits
in the $\Psi$-world' can be developed for any sound 
class $\Psi$. Since lex colimits are used to understand
exactness of enriched categories, one can expect
a theory of `exactness in the $\Psi$-world' for
any sound class $\Psi$. 
This is the matter of future research.

\end{document}